\documentclass[12pt]{article}
\usepackage{latexsym,amsfonts,amssymb,amsthm,amsmath,graphics,epsfig,extarrows}
\usepackage{mathrsfs} 
\usepackage{comment} 
\usepackage{cancel}  
\usepackage[T1]{fontenc}
\usepackage{tikz}

\usepackage{todonotes}

\usepackage{enumitem}  

\usetikzlibrary{arrows,shapes}
\usetikzlibrary{decorations.pathmorphing}
\usetikzlibrary{calc}

\usepackage[margin=.6in]{geometry}

\usepackage[title,titletoc]{appendix}

\usepackage{empheq}
\usepackage{cases}

\usepackage{sectsty}
\sectionfont{\large}

\newtheorem{theorem}{Theorem}[section]
\newtheorem{lemma}[theorem]{Lemma}
\newtheorem{proposition}[theorem]{Proposition}

\theoremstyle{remark}
\newtheorem{remark}{Remark}[section]

\numberwithin{equation}{section}
\numberwithin{theorem}{section}

\usepackage{everysel}
\EverySelectfont{%
	\fontdimen2\font=0.4em
	\fontdimen3\font=0.2em
	\fontdimen4\font=0.1em
	\fontdimen7\font=0.1em
	\hyphenchar\font=`\-
}
\usepackage{tgcursor}
\usepackage{dsfont}
\usepackage{cases}
\usepackage[colorlinks=true]{hyperref}
\hypersetup{urlcolor=blue, citecolor=blue}

\newcommand{\dvv}[1]{{\mbox{div}(#1)}}

\newcommand{\iQ}{\int_Q}

\newcommand{\iD}{\int_{\Sigma_1}}

\begin{document}
\allowdisplaybreaks
	\title{Boundary stabilization  of  the linear MGT equation with  partially 
		absorbing boundary data and  degenerate viscoelasticity}
	\date{}
\maketitle

\vspace{-2cm}

\centerline{\scshape Marcelo Bongarti$^*$}
\medskip
{\footnotesize
 \centerline{Weierstrass Institute for Applied Analysis and Stochastics (WIAS)}
   \centerline{Mohrenstrasse 39, 10117 Berlin, Germany}
} 

\medskip

\centerline{\scshape Irena Lasiecka and Jos\'e H. Rodrigues}
\medskip
{\footnotesize
 \centerline{Department of Mathematical Sciences, The University of Memphis}
   \centerline{IBS,  Polish Academy of Sciences, Warsaw
   }
}
\begin{abstract}The Jordan--Moore--Gibson--Thompson (JMGT) equation is a well-established and recently widely studied model for nonlinear acoustics (NLA). It is a third--order (in time) semilinear Partial Differential Equation (PDE) with a  distinctive feature of predicting the propagation of ultrasound waves at \textit{finite} speed. This is due to the heat phenomenon known as \textit{second sound}  which leads to hyperbolic heat-wave  propagation. In this paper, we consider the problem in the so called "critical" case, where free dynamics is unstable. In order to stabilize, we shall use boundary  feedback controls supported on a portion of the boundary only. Since the remaining part of the boundary is not "controlled", and the imposed boundary conditions of Neumann type  fail to saitsfy Lopatinski condition, several mathematical issues typical for  mixed problems within the context o  boundary stabilizability arise. To resolve these,  special geometric constructs  along with sharp trace estimates will be developed. The imposed  geometric conditions  are motivated by  the 
  geometry that is suitable for modeling  the problem of controlling (from the boundary) the acoustic pressure involved in medical treatments such as lithotripsy, thermotherapy, sonochemistry, or any other procedure involving High Intensity Focused Ultrasound (HIFU).

\end{abstract}

\section{Introduction} It was not until the first decade of the XXI century that third--order in time models became central in the study of the propagation of acoustic waves. Fattorini \cite{fattorini_ordinary_1969}  points out that models with three time derivatives are, in general, ill-posed. Nevertheless, from a modelling point of view, the appearance of a third derivative in time seems unavoidable. For once, if one seeks to understand the effects of (thermal) relaxation in the propagation of sound, a (by now) well--known strategy is the use of hyperbolic models for the heat flux (also known as \textit{second--sound} phenomena), which introduce one extra time derivative \cite{jordan_nonlinear_2008,jordan_second-sound_2014} in explicit models. Even more essential for controlling medical or engineering (acoustic) phenomena is the fact that the presence of a third time derivative predicts finite speed of propagation of the waves, a novelty in comparison to the classic parabolic models where heat fluxes are modeled through diffusion (Fourier's law). The issue of wellposedness is \textit{naturally} remedied in modern formulations of nonlinear acoustics -- in particular models leading to the so called HIFU field -- due to the structural damping effect caused by sound diffusivity in a given tissue or group of tissues \cite{kaltenbacher_wellposedness_2012}. This pleasant feature allows for a better understanding of the role of sound diffusion  and propagation  in the acoustic environment. Instead, classical second--order in time models (see \eqref{modnl2} below)  lead  to strong smoothing of solutions exhibited 
by the analyticity of the underlying dynamics.

Studies toward more accurate calibration of HIFU field generator devices are plenty, specially in the last few decades. Such devices are pivotal for several types of thermal therapy as treatment of ablating solid tumors of the prostate, liver, breast, kidney, brain, among others. The feature of raising the temperature of a focal region very rapidly with minimal damage to the biological material around it comes at the price of very high (sometimes even with formation of shocks) acoustic pressure \cite{chen2014acoustic}. It is, therefore, of paramount interest  the study of models that provide suitable (optimal) profiles for the HIFU devices ensuring that the acoustic pressure will remain within safety range.
In fact, in recent years we have witnessed a large body of work dealing with the questions of wellposedness and stability of third order dynamics, in both linear and nonlinear versions \cite{kaltenbacher_jordan-moore-gibson-thompson_2019,kaltenbacher_mathematics_2015} and on bounded \cite{marchand_abstract_2012} and unbounded ($\mathbb{R}^n$) domains \cite{pellicer_optimal_2019}. However, very little is known regarding how the third order model responds to the {\it inputs from the boundary-particularly with low regularity}. This particular interest needs no defense, due to prevalence of boundary control problems (imaging, HIFU)  associated with acoustic waves that can be actuated just on the boundary of the spatial region.  Since the model itself can be seen as a hyperbolic system \cite{bucci_cauchy-dirichlet_2020} -- which is however {\it characteristic} -- one may expect mathematical  interest and  the associated challenges.  Of great physical and mathematical interest are issues such as wellposedness with low regularity boundary data and a potential  stabilizing effect of boundary  damping. The latter  is particularly of interest in the case when natural viscoelastic damping (strongly compromised  by the second sound phenomenon) is either very weak or even non--existent. 

This paper accomplishes an important  step towards  the described  goal, namely a boundary  stabilizability property for the linearized third--order in time acoustic wave models with degenerated viscoelastic effects and with boundary dissipation located on a suitable portion of the boundary. One of the salient feature is the fact that 
part of the boundary {\it subject to Neumann boundary conditions}  is not observed/dissipated - in line with the configuration expected from applications to  boundary  control. Unobserved Neumann part of the boundary (rather then Dirichlet where suitable  methods have been well  developed)  is known  as causing major challenges in the derivation of observability estimates -- even in the case of the wave equation \cite{LTZ}. This difficulty is dealt with by using suitable geometric and microlocal analysis constructions applicable to the third--order in time models.

\subsection{PDE Model and  Motivation}
We assume that the acoustic pressure $u = u(t,x)$ at the material point $x \in \mathbb{R}^d$ ($d = 2$ or $3$) and instant $t \in \mathbb{R}_+$ obeys the Jordan--Moore--Gibson--Thompson equation 
\begin{equation}\label{modnl} 
	\tau u_{ttt} + (\alpha - 2ku)u_{tt} - c^2 \Delta u - (\delta + \tau c^2)\Delta u_t = 2ku_t^2,
\end{equation} 
where $c, \delta, k > 0$ are constants representing the speed and diffusivity of sound and a nonlinearity parameter, respectively. The function $\alpha: \Omega \to \mathbb{R}$ represents the natural frictional damping provided by the medium. The parameter $\tau >0$ represents the thermal relaxation time and its presence allows for a more precise distinction within propagation of sound in different media. Indeed, the semilinear equation is a (singular perturbation) \textit{refinement} of the classical quasilinear Westervelt's equation ($\tau = 0$): \begin{equation}
	\label{modnl2} (\alpha - 2ku)u_{tt} - c^2 \Delta u - \delta\Delta u_t = 2ku_t^2.
\end{equation} Although not unique, one interesting way of obtaining \eqref{modnl} from a similar procedure as the one to obtain \eqref{modnl2} is simply to use Maxwell--Cattaneo law \cite{ekoue_maxwell-cattaneo_2013,cattaneo_form_1958,cattaneo_sulla_2011} in place of Fourier's law. The advantage of this strategy (which is by no means physics--proof \cite{spigler2020more,christov_heat_2005}) is that it provides a suitable model for studying relaxation effects.Since waves  propagate at a  finite speed, it allows the construction of optimal policies for controlling the HIFU field. Overall, in its simplicity, \eqref{modnl} catches most of the key features that would be present in a more detailed model.

The mathematical study of \eqref{modnl} as well as the differences (and similarities) when compared to \eqref{modnl2} started around 2010 with the works of I. Lasiecka, R. Triggiani and B. Kaltenbacher \cite{kaltenbacher_exponential_2012,kaltenbacher_wellposedness_2011,marchand_abstract_2012} where the issues of wellposedness and stability of solutions under  homogeneous Dirichlet and Neumann boundary data were addressed for both nonlinear and linearized dynamics. The obtained results depend critically on the positivity of the stability parameter 
\begin{equation}\label{gm}
	\gamma(x) \equiv \alpha(x) - \frac{\tau c^2}{b} \geq \gamma_0 > 0 \ \mbox{a.e. in } \Omega.
\end{equation} 
In addition to ensuring uniform exponential decays of solutions for the linearized $(k=0)$ problem, condition \eqref{gm} allows for the construction of nonlinear flows via "barrier's" method. In face of such results, the natural  question is: \textit{what if $\gamma(x)$ is no longer positive?} It is known that if $\gamma < 0 $ one may have chaotic solutions \cite{conejero_chaotic_2015}. If $\gamma \equiv 0$ then the energy is conserved \cite{kaltenbacher_wellposedness_2011,kaltenbacher_wellposedness_2012}. This raises an interesting question  on how to ensure stability of  the dynamics when the frictional  parameter $\gamma $ degenerates  $\gamma(x)  \geq 0 $. 
It has been recently shown that adding viscoelastic effects produces in some cases  the asymptotic decay of the energy, cf. e.g. \cite{lasiecka_mooregibsonthompson_2015,lasiecka_mooregibsonthompson_2016,delloro_fourth-order_2017}. In this work we concentrate on boundary stabilization. This is  also motivated by recent consideration of control problems defined for MGT dynamics \cite{clason,bucci_feedback_2019}. By actuation -- say on the  boundary -- one aims at obtaining a desired outcome measured by certain functional cost. 
It is well known that control problems -- particularly on infinite horizon -- are strongly linked with stabilizability properties of the linearized model. One interesting problem , considered by Clason-Kaltenbacher in  \cite{clason} 
is that of actuating the external part of the boundary through a transducer\footnote{A \textbf{transducer} is a device that takes power from one source and supplies power usually in another form to a second system. In the particular case of HIFU processes, the transducer concentrates the energy generated by the vibration of sound in a given medium and delivers it to a targeted area in form of heat.}  with the aim at targeting acoustic signal on a given area inside the domain, cf. Figure \ref{fig1}. Such configuration will call for stabilizing effects  emanating from uncontrolled part of the boundary, say $\Gamma_1$, while 
the actuation itself will take place on the remaining -- accessible to the user -- part of the boundary, say $\Gamma_0$, which is not subjected to dissipation or absorption. 
This configuration leads to the following model. 

\begin{subnumcases}{}
	\tau u_{ttt} + \alpha(x) u_{tt} - c^2 \Delta u - b\Delta u_t = f & in $Q = (0,T] \times \Omega$ \label{E1-1a} \\[2mm]
	u(0)  = u_0; \quad u_t(0)  = u_1; \quad u_{tt}(0)  = u_2 & in $\Omega$   \label{E1-1b}   \\[2mm]
	\partial_\nu u + \kappa_0(x)u = 0  & in $\Sigma_0 = (0,T] \times \Gamma_0$   \label{E1-1c}  \\[2mm]
	\partial_\nu u + \kappa_1(x)u_t = 0  & in $\Sigma_1 = (0,T] \times \Gamma_1$   \label{E1-1d} 
\end{subnumcases} with $\tau,c,b >0,$ $\alpha \in L^\infty(\Omega)$, $\kappa_0 \in L^\infty(\Gamma_0)$ and $\kappa_1 \in L^\infty(\Gamma_1),$ $\kappa_1(x) \geq \kappa_1 > 0 , \kappa_0 >0$ a.e.

Assuming for the time being that a solution $u = u(t,x)$ for \eqref{E1-1a}--\eqref{E1-1d} exists in a suitable topology, and critical parameter admits degeneracy  $\gamma (x) \geq 0 $,  our goal is to study its asymptotic properties as $t \to \infty$. More precisely, we want to show that for large times the acoustic pressure will be small, i.e., $\lim\limits_{t \to \infty} u(t,\cdot) = 0$, hopefully at exponential rate.

It should be noted that boundary stabilization of linear MGT has been studied recently. However, the existing results \cite{bongarti2020boundary}  and \cite{BLT}  do not allow for un-dissipated $\Gamma_0$ with Neuman-Robin boundary conditions and  degenerate viscoelasticity. The latter provides for major mathematical challenge (even in the case of wave equation). This is due to the fact that boundary conditions on $\Gamma_0$  {\it }fail to satisfy strong Lopatinski conditions. 
On the other hand, control problems under consideration call for $\Gamma_0$ to be an  active (rather than passive) wall  where control actuation takes place. From the mathematical point of view, this new scenario requires drastically different  strategies and constructions. 

This brings us  up to the main topic of this paper:  stabilization problem closely related to optimal control problems for MGT in infinite time horizon and with Neumann boundary  feedback supported only partially on $\Gamma$. In order to motivate our assumptions on the geometry we look at Figure \ref{fig1} where a schematic transducer is represented. The only needed (and realistic) assumption is the convexity of the red part, which we will call $\Gamma_0.$ The other portion of the boundary, $\Gamma_1$,  will be assumed to be "smooth". The schematic representation is given in Figure \ref{fig2}. \begin{figure}[t]
	\centering
	\includegraphics[scale=.25]{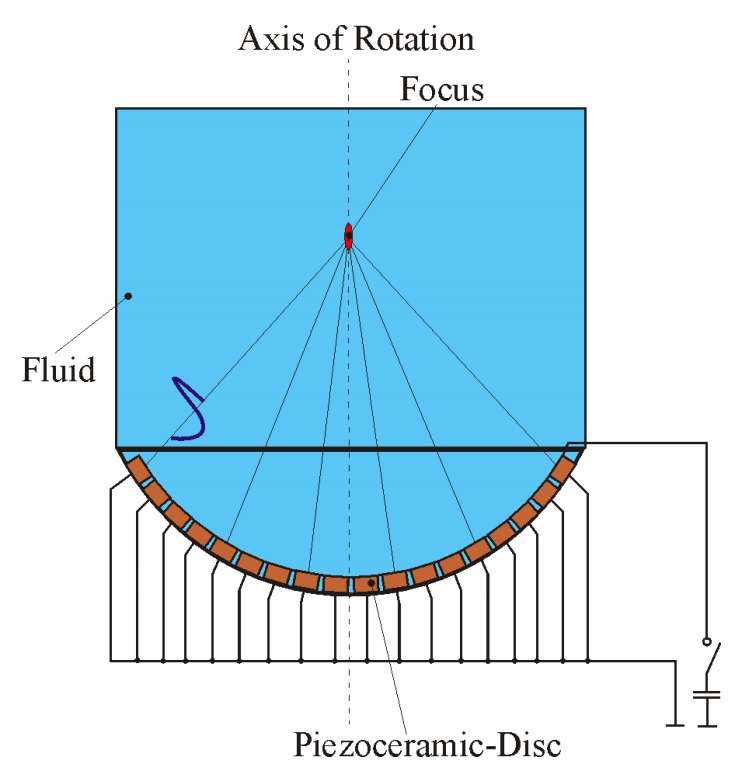}
	\caption{\small Illustration of the domain. The ``red'' convex  portion of the boundary, in the context of HIFU, represents a device called \textit{transducer} and its role is to concentrate the sound waves in the direction of the focus. The remaining part of the boundary represents an absorption area. (Font: B. Kaltenbacher)}
	\label{fig1}
\end{figure}  
\begin{figure}[t]
	\centering
	\includegraphics[scale=.35]{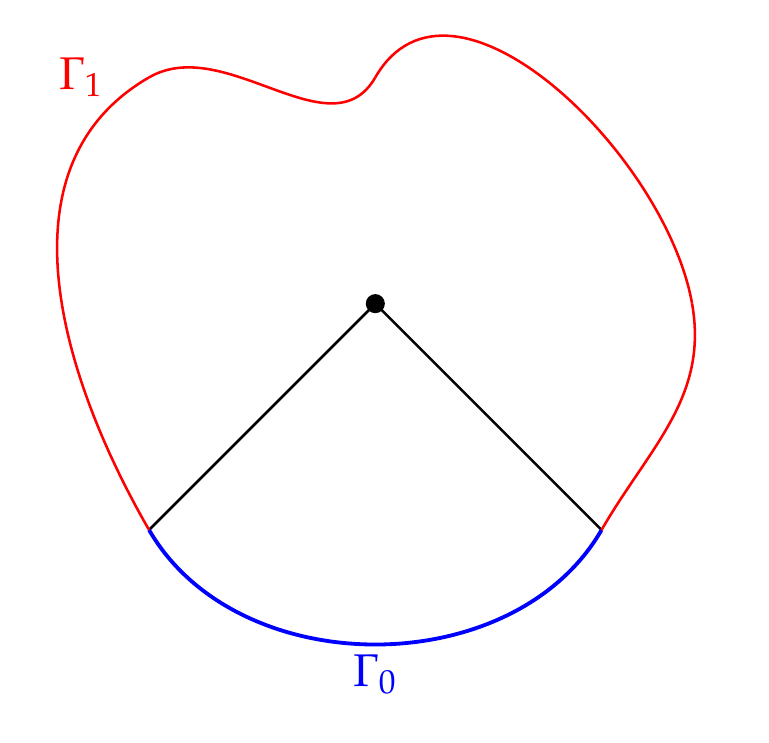}
	\caption{Representation of the domain}
	\label{fig2}
\end{figure}

From the practical point of view, the quantity $\gamma(x)$ is interpreted as the viscoelasticity at the material point $x \in \Omega$ and, in particular in the medical field, is not expected to be known for all points of $\Omega$. By making the more physically relevant assumption that $\gamma \in L^\infty(\Omega)$,  $\gamma(x) \geq 0$ a.e. in $\Omega$ (allowing the critical case $\gamma \equiv 0$, or the case where measurements can only me made at isolated points of the domain), we ask ourselves whether a non--invasive (boundary) action can drive the pressure to zero at large times regardless of the particular knowledge of $\gamma$ (as long as it is nonnegative). This question was answered in \cite{bongarti2020boundary,BLT} with  the final conclusion  that in the case Dirichlet zero boundary conditions are assumed on $\Gamma_0$, which is also star-shaped,  the dissipative boundary effects assumed on $\Gamma_1$ is strong enough to stabilize the system regardless of the particular structure of $\gamma\geq 0$

The present paper addresses the  problem:  what happens when boundary conditions on $\Gamma_0$  are of Neuman-Robin type [Lopatinski condition fails]? This allows to  place actuators on  $\Gamma_1$. Thus  we keep the dissipative Neumann boundary condition on $\Gamma_1$ and supplement $\Gamma_0$ with a  homogeneous Robin boundary data (see \eqref{E1-1c}). Our result states   that uniform stability still holds provided, however, that $\Gamma_0$ is {\it convex} in addition to being star-shaped. If one considers a "benchmark" optimal control problem: \begin{subnumcases}{}
	\min\limits_{g} \ J(u) := \int_0^{\infty}\|\nabla u\|^2_{\Omega} + \|g\|^2_{\Gamma_0} \\\ 
	\mbox{subject to} \ \eqref{E1-1a}, \eqref{E1-1b}, \eqref{E1-1d} \ \mbox{and replacing} \ \eqref{E1-1c} \ \mbox{by} \ \dfrac{\partial u}{\partial \nu}\biggr\rvert_{\Gamma_0} = g
\end{subnumcases} then showing that \eqref{E1-1a}--\eqref{E1-1d} is uniform exponentially stable  proves  a {\it stabilizability} property for the control problem introduced above. Indeed, one takes $ g = -u|_{\gamma_0 } \in L_2(\Gamma_0)$ as a stabilizing feedback. This means   that at least one strategy (control) exists capable of stabilizing the system on infinite horizon. 
We note that related optimal control problem subject to "smooth" controls and finite time horizon has been considered in \cite{clason}. Our point is to address the case of {\it nonsmooth} controls -just  $L^2$ controls- defined on {\it infinite}  time horizon. This leads to new mathematical developments in the area of  boundary stabilizability. 

\section{Main Results} \label{secwpp}
We begin  by introducing a phase (finite energy) space, the abstract version of \eqref{E1-1a}--\eqref{E1-1d} along with some notation. We will work on  a phase space $\mathbb{H}$ given by  \begin{equation}\label{ph-sp}\mathbb{H} := H^1(\Omega) \times H^1(\Omega) \times L^2(\Omega)\end{equation}
To proceed, let  $A: \mathcal{D}(A) \subset L^2(\Omega) \to L^2(\Omega)$ be the operator defined as 
\begin{equation}\label{oplap} 
	A\xi = -\Delta \xi,  \ \ 
	\mathcal{D}(A) = \left\{\xi \in H^2(\Omega); \  \partial_\nu\xi\rvert_{\Gamma_1} =0, \left[\partial_\nu\xi+\kappa_0 \xi\right]_{\Gamma_0}= 0 \right\} 
\end{equation}
In this setting, $A$ is a positive,  self-adjoint operator with compact resolvent and $\mathcal{D}\left(A^{1/2}\right) = H^1(\Omega)$ (equivalent norms). In addition, with some abuse of notation we (also) denote by $A: L^2(\Omega) \to [\mathcal{D}(A)]'$ the extension (by duality) of the operator $A.$

Next, we write \eqref{E1-1a}--\eqref{E1-1d} as a first--order abstract system on $\mathbb{H}$. To this end we need to introduce Harmonic (boundary $\to$ interior) extensions for the Neumann data on $ \Gamma_1$. We proceed as follows: for $\varphi \in L^2(\Gamma_1)$, let $\psi := N(\varphi),$ be the unique solution of the elliptic problem \begin{equation}
	\begin{cases}
		\Delta\psi = 0 \ & \mbox{in} \ \Omega \\ \partial_\nu \psi = \varphi\rvert_{\Gamma_1} \ & \mbox{on} \ \Gamma_1 \\ \partial_\nu\psi +\kappa_0 \psi  = 0 \ & \mbox{on} \ \Gamma_0.
	\end{cases} \label{ep}
\end{equation} It follows from elliptic theory that $N \in \mathcal{L}(H^s(\Gamma),H^{s+3/2}(\Omega))$\footnote{$\mathcal{L}(X,Y)$ denote the space of linear bounded operators from $X$ to $Y$} $(s \in \mathbb{R})$ and 
\begin{equation} \label{neq} 
	N^\ast A \xi = \begin{cases}
		\xi \ \mbox{on} \ &\Gamma_1 \\ 0 \ \mbox{on} \ &\Gamma_0,
\end{cases}
\end{equation} 
for all $\xi \in \mathcal{D}(A)$, where $N^\ast$ represents the adjoint of $N$ when the latter is considered as an operator from $L^2(\Gamma_1) $ to $L^2(\Omega) $. For the reader's convenience, we present a short proof of \eqref{neq} since it will be used critically for the proof of Theorem \ref{t1-1}(i). For $\xi \in \mathcal{D}(A)$ and $\varphi \in L^2(\Gamma_1)$, we first use the definition of adjoint followed by Green's second formula to obtain 
\begin{equation*}
	(N^\ast A \xi,\varphi)_\Gamma = (A\xi,N\varphi) = -(\Delta \xi,N\varphi) = -(\xi,\Delta(N\varphi)) - (\partial_\nu \xi,N\varphi)_\Gamma + (\xi,\partial_\nu N\varphi)_\Gamma.
\end{equation*} 
To complete we use the definition of $N\varphi$ as the solution of problem \eqref{ep} and the fact that $\xi \in \mathcal{D}(A)$. This gives\begin{align}
	\label{a1} (N^\ast A \xi,\varphi)_\Gamma &= -(\xi,\Delta(N\varphi)) - (\partial_\nu \xi,N\varphi)_\Gamma + (\xi,\partial_\nu N\varphi)_\Gamma \nonumber \\ &= (\kappa_0 \xi,N\varphi)_{\Gamma_0} + (\xi,\partial_\nu N\varphi)_{\Gamma_0} + (\xi,\varphi)_{\Gamma_1}\nonumber \\ &= (\xi,\kappa_0N\varphi + \partial \nu N\varphi)_{\Gamma_0} + (\xi,\varphi)_{\Gamma_1} = (\xi,\varphi)_{\Gamma_1},
\end{align} 
which is precisely \eqref{neq}.

Thus, the $u$--problem can be written (distributionally with the values in $[D(A)]' $)  as 
\begin{align}\label{absver}
	\tau u_{ttt} + \alpha(x) u_{tt} + c^2Au + bAu_t + c^2 AN(\kappa_1N^*Au_t) + bAN(\kappa_1N^*Au_{tt}) = f.
\end{align} 

Next, we introduce the operator $\mathscr{A}: \mathcal{D}(\mathscr{A}) \subset \mathbb{H}  \to \mathbb{H}$ with  the action (on $\vec{\xi} = (\xi_1, \xi_2, \xi_3)^\top$):
\begin{equation}\label{opuA}
	\mathscr{A}\vec{\xi} := (\xi_2,\xi_3,-\alpha\xi_3 - c^2A(\xi_1 + N(\kappa_1N^*A\xi_2)) - bA(\xi_2 + N(\kappa_1N^*A\xi_3)))
\end{equation} and  the domain 
\begin{align}
		\label{domA} \mathcal{D}(\mathscr{A}) &= \left\{\vec{\xi} \in \mathbb{H};    \ \xi_3 \in \mathcal{D}\left(A^{1/2}\right), \ \xi_i + N(\kappa_1N^*A\xi_{i+1})\in D(A),~\mbox{for}~i=1,2\right\} \nonumber \\ &= \left\{\vec{\xi} \in \left[H^2(\Omega)\right]^2 \times \mathcal{D} \left(A^{1/2}\right); \left[\dfrac{\partial \xi_1}{\partial \nu} + \kappa_0\xi_1\right] _{\Gamma_0}\!\!\!\!\!= 
		\left[\dfrac{\partial \xi_2}{\partial \nu} + \kappa_0\xi_2\right]_{\Gamma_0}\!\!\!\!\! =0 \right. \nonumber \\ & \hspace{4.8cm}\left. \left[\dfrac{\partial \xi_1}{\partial \nu} + \kappa_1\xi_2\right] _{\Gamma_1}\!\!\!\!\!= 
		\left[\dfrac{\partial \xi_2}{\partial \nu} + \kappa_1\xi_3\right]_{\Gamma_1}\!\!\!\!\! =0
		\right\}.
\end{align}
The first order abstract version of $u$--problem is thus given by
\begin{equation} \begin{cases}
		\label{usist} \Phi_t = \mathscr{A}\Phi + F\\ 
		\Phi(0) = \Phi_0 = (u_0,u_1,u_2)^\top,\end{cases}
\end{equation} 
with $\mathscr{A}:\mathcal{D}(\mathscr{A}) \subset \mathbb{H} \to \mathbb{H}$
defined in (\ref{opuA}) and $F^\top =(0,0,f).$

We are ready for our first result. 
\begin{theorem}{\bf [Wellposedness and Regularity]}\label{t1-1}
	\begin{itemize} 
		\item[\bf (i)] The operator $\mathscr{A}$ generates a strongly continuous semigroup on $\mathbb{H}$.
		\item[\bf (ii)] Let $f \in L^1(0,T,L^2(\Omega))$,  $\kappa_1,\kappa_2 \geq 0 $ a.e and $\alpha \in L^\infty(\Omega)$. Then, for  every initial data $\Phi_0 :=(u_0,u_1,u_2)$ in $\mathbb{H}$,
		there exists a unique (semigroup) solution $\Phi =(u,u_t,u_{tt}) $  of (\ref{E1-1a})-(\ref{E1-1d}) such that 
		$\Phi \in C([0,T], \mathbb{H}) $ for every $T > 0.$ Moreover, if the initial datum belongs to $\mathcal{D}(\mathscr{A})$ and $f \in C^1([0,T],L^2(\Omega))$ the corresponding solution is in $C((0,T]; \mathcal{D}(\mathscr{A})) \cap C^1([0,T], \mathbb{H})$.
	\end{itemize}
\end{theorem}

Our main result pertains to exponential decay of solutions asserted by Theorem \ref{t1-1} and requires geometric assumptions on the undissipated part of the boundary $\Gamma_0$. 
We assume that $\Gamma_0$ is convex, in the sense of being described by  the level set of a convex function. In addition, we require that the following "star shaped" condition holds: 
\begin{equation}\label{star}
	(x-x_0)\cdot \nu \leq 0
\end{equation} on $\Gamma_0$ for some $x_0 \in \mathbb{R}^n.$
\begin{theorem}{\bf [Uniform stability]} \label{thm34n}
	Let $\gamma (x) \geq 0$ and the geometric condition stated above holds true. The semigroup generated by $\mathscr{A} $ is exponentially stable, i.e., there exist constants $M \geq 1, \omega > 0 $ such that 
	$$\|\Phi(t)\|_{\mathbb{H}} \leq M e^{-\omega t} \|\Phi_0\|_{\mathbb{H}}$$ 
	for all $\Phi_0 \in \mathbb{H} $.
\end{theorem}

\begin{remark} 
	We note critical role being played by the fact that boundary conditions  imposed are of Neumann type and there is nontrivial part of the boundary $\Gamma_0$ which is not dissipated. It is known from the  observability theory for wave equation, that standard techniques  do not apply to uncontrolled Neumann parts of the boundary-\cite{LTZ,LL_2002_NA} . The reason is that known multipliers do not collect the energy from  this part of the boundary due to conflicting sign  of the vector field on the undissipated part of the boundary. Recently, new geometric methods have been introduced in order to handle this difficulty. We shall adapt these methods to the present system. 
\end{remark}

{\bf Conclusion.} The result of Theorem \ref{thm34n} provides a positive answer to the question of exponential stabilizability of MGT equation in the  critical case ($\gamma \equiv 0$) via a  boundary  feedback supported only partially  on $\Gamma$ with the requirement of convexity imposed on $\Gamma_0$. 

The remainder of this paper is devoted to the proofs of  our  main results.

\section{Semigroup Generation}

We recall that, topologically, the space $\mathbb{H}$ introduced in the previous section is equivalent to 
\begin{equation}\label{lowen}
	\mathcal{D}\left(A^{1/2}\right) \times \mathcal{D}\left(A^{1/2}\right) \times L^2(\Omega)\end{equation} with the topology induced by the inner product defined, for all $\vec{\xi} = (\xi_1,\xi_2,\xi_3)^\top, \vec{\varphi} = (\varphi_1,\varphi_2,\varphi_3)^\top \in \mathbb{H}$, as
\begin{equation}
	\label{inH} \left(\vec{\xi},\vec{\varphi}\right)_{\mathbb{H}} = (A^{1/2} \xi_1,A^{1/2} \varphi_1) + b(A^{1/2}\xi_2,A^{1/2} \varphi_2) +
	(\xi_3,\varphi_3).
\end{equation}  Because of this equivalence we will be using the same $\mathbb{H}$ to denote both spaces.
It is useful to notice that 
\begin{equation}\label{norm}
(A^{1/2}u,A^{1/2}v) = (\nabla u, \nabla v) + \kappa_0 \int_{\Gamma_0} u v d\Gamma_0 
\end{equation}

We need to show that $\mathscr{A}: \mathcal{D}(\mathscr{A}) \subset \mathbb{H} \to \mathbb{H}$ generates a strongly continuous semigroup on $\mathbb{H}.$ It is convenient to introduce the following  change of variables $bz = bu_t + c^2 u$ (see \cite{marchand_abstract_2012}) which reduces the problem to a PDE--abstract ODE coupled system.

Let $M \in \mathcal{L}(\mathbb{H})$ defined by $$M\vec{\xi} = \left(\xi_1,\xi_2 + \dfrac{c^2}{b}\xi_1, \xi_3 + \dfrac{c^2}{b}\xi_2\right)$$ which has inverse $M^{-1}\in \mathcal{L}(\mathbb{H})$ given by $$M^{-1}\vec{\xi} = \left(\xi_1,\xi_2-\dfrac{c^2}{b}\xi_1,\xi_3 - \dfrac{c^2}{b}\xi_2 + \dfrac{c^4}{b^2}\xi_1\right)$$ and therefore is an isomorphism of $\mathbb{H}$. The next lemma makes precise the translation of $u$--problem to a different  system involving a component of  a suitable wave equation labeled by  $z.$

\begin{lemma} \label{equiv_prob}  
	Assume that the compatibility conditions 
	\begin{equation}\label{comp} 
		\dfrac{\partial}{\partial \nu} u_0 + \kappa_0 u_0 = 0 \ \emph{\mbox{on}} \ \Gamma_0, \qquad \dfrac{\partial}{\partial \nu} u_0 + \kappa_1 u_1 = 0 \ \emph{\mbox{on}} \ \Gamma_1
	\end{equation} hold.
	Then $\Phi \in C^1(0,T;\mathbb{H})\cap C(0,T;\mathcal{D}(\mathscr{A}))$ is a strong solution for \eqref{usist} if, and only if, $\Psi= M\Phi \in C^1(0,T;\mathbb{H})\cap C(0,T;\mathcal{D}(\mathbb{A}))$ is a strong solution for  
	\begin{equation} 
		\begin{cases}
			\label{usistz} \Psi_t = \mathbb{A}\Psi + G\\ 
			\Psi(0) = \Psi_0 = M\Phi_0 =  \left(u_0,u_1 + \dfrac{c^2}{b}u_0,u_2 + \dfrac{c^2}{b}u_1\right)^\top,
		\end{cases}
	\end{equation} 
	where $G = MF$ and $\mathbb{A} = M\mathscr{A}M^{-1}$ with 
	\begin{equation}\label{domz}
		\begin{split}
		\mathcal{D}(\mathbb{A}) = \left\{\vec{\xi} \in \left[H^2(\Omega)\right]^2 \times \mathcal{D} \left(A^{1/2}\ \right);
		\left[\dfrac{\partial \xi_2}{\partial \nu} + \kappa_0\xi_2\right]_{\Gamma_0} =0, \right. \\ \left. \left[\dfrac{\partial \xi_2}{\partial \nu} + \kappa_1\xi_3\right]_{\Gamma_1} =0
		\right\}
		\end{split}
	\end{equation}
\end{lemma}

\begin{proof}
	The only non-trivial step is to prove that boundary conditions (going from $z$ to $u$--problem) match. To this end, assume that $\Psi = (u,z,z_t) \in C^1(0,T;\mathbb{H}) \cap C(0,T;\mathcal{D}(\mathbb{A}))$ is a strong solution for \eqref{usistz}. Let 
	$$\Upsilon(t) := \left(\dfrac{\partial u(t)}{\partial \nu} + \kappa_0 u(t)\right)\biggr\rvert_{\Gamma_0}, \ t \geq 0$$ 
	and notice that $b\Upsilon_t +c^2\Upsilon = 0$ for all $t$. This along with the compatibility condition \eqref{comp}$_1$ ($\Upsilon(0) = 0$) implies that $\Upsilon \equiv 0$. The same argument \emph{mutatis mutandis} recovers the boundary condition for $u$ on $\Gamma_1.$ The proof is then complete.
\end{proof}

For $\vec{\xi} = (\xi_1,\xi_2,\xi_3)^\top \in \mathcal{D}(\mathbb{A})$ a basic algebraic computation yields the explicit formula for $\mathbb{A}.$
\begin{align}
	\mathbb{A}\vec{\xi} = \left(\xi_2-\dfrac{c^2}{b}\xi_1,\xi_3,-\gamma\left(\xi_3 - \dfrac{c^2}{b}\xi_2 + \dfrac{c^4}{b^2}\xi_1\right) - bA\xi_2 -b\kappa_1ANN^\ast A\xi_3 \right)
\end{align} 
where $\gamma = \alpha - \dfrac{c^2}{b} \in L^\infty(\Omega).$

We are ready for our generation result.

\begin{theorem}\label{gener}
	The operator $\mathscr{A}$ generates a strongly continuous semigroup on $\mathbb{H}$.
\end{theorem}
\begin{proof}
	Equivalently, we show that $\mathbb{A}$ generates a strongly continuous semigroup on $\mathbb{H}$. If $\{S(t)\}_{t \geq 0}$ is the said semigroup then $\{T(t)\}_{t\geq 0}$, $T(t) := M^{-1}S(t)M, t\geq 0$, will be the semigroup generated by $\mathscr{A}.$
	
	Write $\mathbb{A} = \mathbb{A}_d + P$ where $$P\vec{\xi} = \left(\xi_2,0,\dfrac{\gamma c^2}{b}\left(\xi_2- \dfrac{c^2}{b}\xi_1\right)+(1-\gamma)\xi_3\right), \ \ \vec{\xi} \in \mathbb{H}$$ is bounded in $\mathbb{H}$ and 
	\begin{equation}
		\mathbb{A}_d\vec{\xi} = \left(- \dfrac{c^2}{b}\xi_1,\xi_3,-\xi_3 - bA(\xi_2+N(\kappa_1N^*A\xi_3))\right), \ \vec{\xi} \in \mathcal{D}(\mathbb{A}_d),
	\end{equation} 
	where $\mathcal{D}(\mathbb{A}_d) := \mathcal{D}(\mathbb{A}).$ It then suffices to prove generation of $\mathbb{A}_d$ on $\mathbb{H}$, see \cite[Page 76]{pazy_semigroups_1992}
	
	We start by showing dissipativity: for $\vec{\xi}=(\xi_1,\xi_2,\xi_3)^\top \in \mathcal{D}(\mathbb{A})$ we have \begin{align*}
		\left(\mathbb{A}_d\vec{\xi},\vec{\xi}\right)_{\mathbb{H}} &= - \dfrac{c^2}{b}\|A^{1/2} \xi_1\|_{L^2(\Omega)}^2 +  b\left(A^{1/2} \xi_3, A^{1/2}\xi_2\right)  \\ & -\|\xi_3\|_{L^2(\Omega)}^2 - b(A^{1/2}\xi_2,A^{1/2}\xi_3)  -b\|\sqrt{\kappa_1}\xi_3\|_{L^2(\Gamma_1)}^2 \\ &= - \dfrac{c^2}{b}\|A^{1/2} \xi_1\|_{L^2(\Omega)}^2 -\|\xi_3\|_{L^2(\Omega)}^2 -b\|\sqrt{\kappa_1}\xi_3\|_{L^2(\Gamma_1)}^2 \leq 0,
	\end{align*} hence, $\mathbb{A}_d$ is dissipative in $\mathbb{H}$.
	
	For maximality in $\mathbb{H}$, given any $L = (f,g,h) \in \mathbb{H}$ we need to show that there exists $\Psi = (\xi_1,\xi_2,\xi_3)^\top \in \mathcal{D}(\mathbb{A})$ such that $(\lambda - \mathbb{A}_d)\Psi = L$, for some $\lambda >0.$ This leads to  a solvability of the system of equations: \begin{equation}\label{systm-max}
		\begin{cases}
			\lambda \xi_1 + \dfrac{c^2}{b}\xi_1 = f, \\ \lambda \xi_2 - \xi_3 = g, \\
			\lambda \xi_3 +\xi_3 + b A(\xi_2+ N(\kappa_1N^*A\xi_3))= h,
		\end{cases}
	\end{equation} which implies $\xi_1 = \left(\lambda + \dfrac{c^2}{b}\right)^{-1}f \in \mathcal{D}(A^{1/2}).$ Moreover, since $A^{-1} \in \mathcal{L}({L^2(\Omega)})$ a combination of the second and third equations above yields \begin{equation}
		\label{sol2} K_\lambda \xi_3 = \lambda A^{-1}h - bg
	\end{equation} where $K_\lambda: L^2(\Omega) \to L^2(\Omega)$ acts on an element $\xi \in L^2(\Omega)$ as $$K_\lambda \xi = \left[(\lambda^2 + \lambda)A^{-1}+b(I+\lambda N(\kappa_1N^\ast A)\right]\xi.$$ 
	
	We now notice the restriction $K_\lambda\rvert_{\mathcal{D}(A^{1/2})}$ is strictly positive. Indeed it follows by \eqref{neq} that, given $\xi \in \mathcal{D}(A^{1/2})$ we have 
		\begin{align*}
			(K_\lambda \xi, \xi)_{\mathcal{D}(A^{1/2})} &= (\lambda^2 + \lambda)\|\xi\|_2^2 + b\|A^{1/2}\xi\|_2^2 + b\lambda(A^{1/2}N(\kappa_1 N^\ast A)\xi,A^{1/2}\xi) \\ &= (\lambda^2 + \lambda)\|\xi\|_2^2 + b\|A^{1/2}\xi\|_2^2 + b\lambda\|\sqrt{\kappa_1}N^\ast A\xi\|_{\Gamma_1}^2 > 0.
		\end{align*} 
	
	Among the consequences of positivity, is the fact that 
	${K_\lambda}_{\mathcal{D}(A^{1/2})}^{-1} \in \mathcal{L}(\mathcal{D}(A^{1/2}))$. Therefore, since $\lambda A^{-1}h - bg \in \mathcal{D}(A^{1/2})$ we have that $$\xi_3 := K_\lambda^{-1}(\lambda A^{-1}h - bg) \in \mathcal{D}(A^{1/2})$$ is the solution of \eqref{sol2}. Finally, $$\xi_2 = \lambda^{-1}(\xi_3+g)  \in \mathcal{D}(A^{1/2}).$$
	
	For the final step to conclude membership of $(\xi_1,\xi_2,\xi_3)$ in $\mathcal{D}(\mathbb{A})$ we look at the \textit{abstract} version of the description of $\mathcal{D}(\mathbb{A})$: 
	\begin{equation*}
		\mathcal{D}(\mathbb{A}) = \left\{\vec{\xi} \in \mathbb{H}; \ \xi_2+N(\kappa_1N^*A\xi_3) \in \mathcal{D}(A)
		\right\}
	\end{equation*} 
	whereby one only needs to check that $\xi_2+N(\kappa_1N^*A\xi_3) \in \mathcal{D}(A)$ since the regularity for the triple $(\xi_1,\xi_2,\xi_3)$ to belong to $\mathbb{H}$ was already established. The desired regularity will follow from \eqref{systm-max}, which implies 
	\begin{align*}
		b\lambda (\xi_2+N(\kappa_1N^*A\xi_3)) &= - (\lambda^2 + \lambda)A^{-1}\xi_3 + \lambda A^{-1} h \in \mathcal{D}(A),
	\end{align*} since $\xi_3, h \in L^2(\Omega).$ 
	
	The proof is complete. 
\end{proof}

Theorem \ref{t1-1}(b) then follows as a straightforward corollary of Theorem \ref{gener} and Lemma \eqref{equiv_prob}.

\section{Stabilization} \label{secsta}

Our stability results rely on a chain of  estimates developed through the   process of the proof whose main ingredient is to propagate dissipation from a portion of the boundary $\Gamma_1$ into the entire domain. Let us  first  outline the main  conceptual ideas. 
\begin{itemize}
	\item[\bf (i)] In order to  handle  the estimates  on  $\Gamma_0$ -- the undissipated part of the boundary -- a typical radial vector field leads to conflicting signs in front of tangential boundary-time  derivative. In order to handle this, special vector fields  are introduced  which are constructed locally by "bending " tangentially  the radial field  on the  undissipated part of the   boundary. This can be accomplished by exploiting convexity of $\Gamma_0$  along with a general star shaped requirement. Having a vector field which is tangential to the boundary allows us to annihilate the normal component of this vector field -- taking care of the tangential derivatives on the undissipated part of the boundary (note that in the Dirichlet case, the contribution on the undissipated part of the boundary is just zero).
	\item[\bf (ii)] On the absorbing part of the boundary $\Gamma_1$  we use the fact that  the time derivative of the  solution is given through the energy relation. By applying microlocal analysis argument, one estimates  the space-time tangential contribution of the solution in terms of the time  time derivatives and some lower order terms. \item[\bf (iii)] The resulting lower order terms are eliminated by  a suitable compactness- uniqueness argument. 
\end{itemize}

We start by introducing the energy functional. We work with smooth (classical solutions) guaranteed by Theorem \ref{t1-1} and then Theorem \ref{thm34n} is obtained via density argument along with the  convexity of the energy functional.

Let $(u,u_t,u_{tt}) \in \mathbb{H}$ be a classical solution of \eqref{usist} and recall the corresponding $z$--problem determined via \eqref{usistz}, from which follows that $z = u_t + \frac{c^2}{b}u$ solves the equation
\begin{equation}\label{zprob}
	z_{tt} + bA(z_t + N(\kappa_1N^*Az)) = -\gamma u_{tt} + f,
\end{equation} with initial conditions described in \eqref{usistz}.

With this notation, we define the functional $E(t) = E_0(t) + E_1(t)$ where $E_i:[0,T] \to \mathbb{R}_+$ ($i = 0,1$) are defined by 
\begin{align}\label{E1} 
	E_1(t) & :=
	\dfrac{b}{2}\|A^{1/2}{z}\|_2^2 +  \dfrac{1}{2}\|z_t\|_2^2 + \dfrac{c^2}{2b}\|\gamma^{1/2}u_t\|_2^2
\end{align}  
and 
\begin{align}\label{E0} 
	E_0(t)& := \dfrac{1}{2}\|\alpha^{1/2}u_t\|_2^2 + \dfrac{c^2}{2}\|A^{1/2}{u}\|_2^2.
\end{align}

The next lemma guarantees that stability of solutions in $\mathbb{H}$ is equivalent to uniform exponential decay of the function $t \mapsto E(t).$
One thing to notice is that $E_1(t) $ is dissipative along the  unforced solution.This  no longer holds  for the full energy $E(t)$.  
\begin{lemma}\label{equinorm}
	Let $\Phi = (u,u_t,u_{tt})$ be a weak solution for the $u$-- problem in $\mathbb{H}$ and assume that \eqref{comp} is in force. Then the following statements are equivalent: 
	\begin{itemize}
		\item[\bf a)] $t \mapsto \|\Phi(t)\|_{\mathbb{H}}^2$ decays exponentially.
		
		\item[\bf b)] $t \mapsto \|M\Phi(t)\|_{\mathbb{H}}^2 = \|(u,z,z_t)\|_{\mathbb{H}}^2$ decays exponentially.
		
		\item[\bf c)] $t \mapsto E(t)$ decays exponentially.
	\end{itemize}
\end{lemma} 
\begin{proof} 
Proof relies on algebraic manipulations. Details can be found in 
	\cite{bongarti2020boundary}.
\end{proof}

\begin{remark}
	The purpose of Lemma \eqref{equinorm} is that it allows us to use both the expression of the energy $E(t) = E_0(t) + E_1(t)$ and the norm of the solution $\|(u,z,z_t)\|_{\mathbb{H}}^2$ interchangeably. The specific structure of the energy contributes to a discovery of certain invariances and dissipative laws. 
	However, from the topological point of view, it is essential that 
	the following three quantities $\|A^{1/2} z\|_2$, $\|z_t\|_2$ and $\|\nabla u\|_2$ 
	display  the appropriate decays. 
\end{remark}
The next proposition provides the set of main identities for the linear stabilization in $\mathbb{H}.$
\begin{proposition}\label{id}
	Let $T>0$. If $(u,z,z_t)$ is a classical solution of \eqref{usistz} then the following holds
	\begin{itemize}
		\item[\bf (i) ]{\bf{(Energy Identity)}} For $0\leq t\leq T$, 
		\begin{equation}\label{e1id}
			E_1(T) 
			+ b\int_t^T\int_{\Gamma_1} \kappa_1z_t^2d\Gamma_1ds  
			+ \int_t^T\int_\Omega \gamma u_{tt}^2 d\Omega ds
			= E_1(t) + \int_t^T\int_\Omega fz_t d\Omega ds.
		\end{equation}
		\item[\bf (ii)] {\bf{Energy ($L^1$--norm)--Reconstruction.}} For $0<s<T/2$, 
		\begin{equation}\begin{aligned}\label{e1rec}
				\int_s^{T-s} E_1(t) dt 
				\lesssim [E_1(s)+E_1(T-s)]
				\\ + C_T\left[ b\iD \kappa_1z_t^2d\Sigma_1 + \iQ \gamma u_{tt}^2dQ
				+ lot_\delta(z) \right]
				+ C_T\int_Q f^2 dQ.
		\end{aligned}\end{equation}
		where $lot_\delta(z)\leq C_\delta\sup_{t\in[0,T]}\left\{\|z\|_{H^{1-\delta}(\Omega)}^2+\|z_t\|_{H^{-\delta}(\Omega)}^2\right\}$, for $\delta>0$.
	\end{itemize}
\end{proposition} 

\begin{remark}
	Notice that the energy identity (\ref{e1id}) involves only partial 
	information on the dynamics. In fact $E_1(t) $ does not reconstruct $\|\nabla u\|_2$, a critical ingredient of the MGT system. The second inequality \eqref{e1rec} is an "almost" reconstruction of partial energy in terms of the dissipation and lower order terms. 
\end{remark}

\begin{proof}
	\textbf{1. Proof of \eqref{e1id}.} Let, on $\mathbb{H}$, the bilinear form $\langle\cdot,\cdot\rangle$ be given by
	\begin{equation}\begin{aligned}\label{biform}
			\langle \vec{\xi},\vec{\varphi}\rangle = 
			b\left(A^{1/2}\xi_2,A^{1/2}\varphi_2\right)
			+ \left(\xi_3,\varphi_3\right)
			+ \frac{c^2}{b}\left(\gamma\left(\xi_2-\frac{c^2}{b}\xi_1\right),\varphi_2-\frac{c^2}{b}\varphi_1\right)
	\end{aligned}\end{equation} for all $\vec{\xi} = (\xi_1,\xi_2,\xi_3)^\top, \vec{\varphi} = (\varphi_1,\varphi_2,\varphi_3)^\top \in \mathbb{H}$,
	which is continuous. Moreover, recalling that $\Psi(t)=(u(t),z(t),z_t(t))$ it follows that $2E_1(t) = \langle\Psi(t),\Psi(t)\rangle$. Therefore,
	\begin{align*}
		\dfrac{dE_1(t)}{dt} 
		&= \left\langle\dfrac{d\Psi(t)}{dt},\Psi(t)\right\rangle
		= \left\langle\mathbb{A}\Psi(t) + G,\Psi(t)\right\rangle \\
		&= \left\langle \!\!\! \left(z-\frac{c^2}{b}u,z_t,-\gamma\!\left(z_t-\frac{c^2}{b}z+\frac{c^4}{b^2}u\right)\!-bA(z_t+N(\kappa_1N^*Az))+f\right)^T\!\!\!\!,\Psi(t)\!\!\! \right\rangle \\
		&= 
		\left(-\gamma\left(z_t-\frac{c^2}{b}z+\frac{c^4}{b^2}u\right)-bA(z_t+N(\kappa_1N^*Az))+f,z_t\right)\\
		&\quad +b\left(A^{1/2}z_t,A^{1/2}z\right)+ \frac{ c^2}{b}\left(\gamma\left(z_t-\frac{c^2}{b}\left(z-\frac{c^2}{b}u\right)\right),z-\frac{c^2}{b}u\right) \\
		&= 
		-\left(\gamma\left(z_t-\frac{c^2}{b}z+\frac{c^4}{b^2}u\right),z_t\right)
		-b(A(z_t+N(\kappa_1N^*Az)),z_t)\\
		&\quad +b\left(A^{1/2}z_t,A^{1/2}z\right)+ (f,z_t)
		+\frac{ c^2}{b}\left(\gamma\left(z_t-\frac{c^2}{b}\left(z-\frac{c^2}{b}u\right)\right),z-\frac{c^2}{b}u\right)\\ &= -\int_\Omega\!\!\!\gamma \left(z_t-\frac{c^2}{b}z+\frac{c^4}{b^2}u\right)\left(z_t-\frac{c^2}{b}z+\frac{c^4}{b^2}u\right) d\Omega
		- b\int_{\Gamma_1}\!\!\!\kappa_1z_t^2d\Gamma_1 + (f,z_t) \\ &= -\int_\Omega \gamma u_{tt}^2d\Omega - b \int_{\Gamma_1}\kappa_1z_t^2d\Gamma_1 + \int_\Omega fz_td\Omega
	\end{align*} since $z_t-\frac{c^2}{b}z+\frac{c^4}{b^2}u = u_{tt}$. Identity \eqref{e1id} then follows by an integration in time on $(t,T).$
	
	\
	
	\textbf{2. Proof of \eqref{e1rec}}
	This second part will be established via multipliers technique making strong use of the geometrical conditions preceding the statement of Theorem \ref{thm34n}. However, 
	due to the fact that Neumann boundary conditions are imposed on the acoustic pressure $u$ and the absorbtion of the energy (dissipation) occurs only on a portion ($\Gamma_1$) of it, standard radial multipliers used in observability  theory of waves do not apply. There is a conflicting sign requirement for the vector field to be constructed \cite{LTZ,LL_1999_CC}. To resolve this issue one needs to construct a different multiplier with the property that its Jacobian  generates a positive metric, and at the same time complies with the conflicting sign on $\Gamma_1$. This has been accomplished (see for instance \cite{LTZ}) under the condition that the non-dissipative part of the boundary is conve and  leads to a construction  of a special vector field $h\in C^1$ which enjoys the following properties \cite{LL_1999_CC}
	
	\begin{equation*}
		h\cdot\nu = 0~\mbox{on}~\Gamma_0,\quad J(h(x))\geq c_0 > 0, x\in \Omega 
	\end{equation*}
	where $J(h)$ denotes the Jacobian matrix of  the vector field $h$. 
	Such vector field has been constructed in \cite{LTZ} based on the idea introduced in \cite{tataru_regularity_1998} and further generalised in \cite{LL_1999_CC} for domains $\Omega$ with the properties: 
	$\Gamma_0 $ is a {\it convex part}  of $\Gamma$ which also satisfies star shaped condition:
	\begin{equation}\label{star2}
		(x-x_0 )\cdot \nu \leqslant 0, \ \mbox{on} \ \Gamma_0, \ \mbox{for some} \ x_0 \in \mathbb{R}^n
	\end{equation}
	
	The  condition  \eqref{star2}  guarantees that a sufficiently large  portion of the boundary  $\Gamma$ is under absorbtion. This is typical condition  required by  Moravetz-Strauss  theory. However, convexity of $\Gamma_0 $ is a new requirement. 
	This allows for a construction of suitable vector field with  the postulated properties.  The construction is based on a  perturbation [bending tangentially]  of the radial vector field. 
	With that field $h$ in hand, we first multiply equation \eqref{zprob} by $h\cdot\nabla{z}$ integrate by parts in $(s,T-s)\times\Omega$. This gives
	\begin{align}
		&\dfrac{b}{2}\int_s^{T-s}\!\!\!\!\int_\Omega J(h)|\nabla z|^2d\Omega dt + \dfrac{1}{2}\int_s^{T-s}\!\!\!\!\int_\Omega \left(z_t^2 - b|\nabla z|^2\right)\dvv{h}d\Omega dt = \nonumber  \\ &- \int_s^{T-s}\!\!\!\!\int_\Omega \gamma u_{tt} (h\cdot\nabla z)d\Omega dt - \int_\Omega z_t(h\cdot\nabla z)d\Omega\biggr\rvert_s^{T-s} \nonumber \\ &+ \dfrac{1}{2}\int_s^{T-s}\!\!\!\!\int_{\Gamma_1}\!\!\!\! \left(z_t^2 - b|\nabla z|^2\right)(h\cdot\nu) d\Gamma_1 dt + \dfrac{1}{2}\int_s^{T-s}\!\!\!\!\int_{\Gamma_0} \!\!\!\!\left(z_t^2 - b|\nabla z|^2\right)(h\cdot\nu) d\Gamma_0 dt
	\label{cancel} \\ 
		&	+ b\int_s^{T-s}\!\!\!\!\int_{\Gamma} \partial_\nu z(h\cdot\nabla z) d\Gamma dt + \int_s^{T-s}\!\!\!\!\int_\Omega f(h\cdot\nabla z)d\Omega dt. \label{hgradz}
	\end{align} where we notice that the second term in \eqref{cancel} vanishes since $h \cdot \nu = 0$ on $\Gamma_0.$
	
	Next, we multiply equation \eqref{zprob} by $z\mbox{div}(h)$ and integrate by parts in $(s,T-s)\times\Omega$. This leads to
	\begin{align}
		&\dfrac{1}{2}\int_s^{T-s}\!\!\!\!\int_\Omega \left(b|\nabla z|^2 - z_t^2\right)\mbox{div}(h)d\Omega dt = \dfrac{b}{2} \int_s^{T-s}\!\!\!\!\int_\Gamma \partial_\nu z z \mbox{div}(h)d\Gamma dt\nonumber \\ &- \dfrac{1}{2}\int_s^{T-s}\!\!\!\!\int_\Omega \gamma u_{tt} z \mbox{div}(h) d\Omega dt - \dfrac{1}{2} \int_\Omega z_tz \mbox{div}(h)d\Omega\biggr\rvert_s^{T-s} \nonumber \\ &- \dfrac{b}{2} \int_s^{T-s}\!\!\!\!\int_\Omega z \nabla z \cdot \nabla(\mbox{div}(h))d\Omega dt + \dfrac{1}{2}\int_s^{T-s}\!\!\!\!\int_\Omega fz\dvv{h}d\Omega dt \label{zdivh} .
	\end{align}
	
	Adding \eqref{hgradz} with \eqref{zdivh} we have \begin{align}
		&\dfrac{b}{2}\int_s^{T-s}\!\!\!\!\int_\Omega J(h)|\nabla z|^2d\Omega dt = \\
		&- \int_s^{T-s}\!\!\!\!\int_\Omega \gamma u_{tt} A_h(z)d\Omega dt - \int_\Omega z_tA_h(z)d\Omega\biggr\rvert_s^{T-s} + \int_s^{T-s}\!\!\!\!\int_\Omega fA_h(z)d\Omega dt\nonumber \\ &+ \dfrac{1}{2}\int_s^{T-s}\!\!\!\!\int_{\Gamma_1} \left( z_t^2 - b|\nabla z|^2 \right)(h\cdot\nu) d\Gamma_1 dt + b\int_s^{T-s}\!\!\!\!\int_\Gamma \partial_\nu zA_h(z)d\Gamma dt \nonumber\\
		&- \dfrac{b}{2} \int_s^{T-s}\!\!\!\!\int_\Omega z \nabla z \cdot \nabla(\mbox{div}(h))d\Omega dt \label{est1}
	\end{align} 
	where $A_h(z) = h\cdot\nabla z + \dfrac{1}{2}z\dvv{h}$. Notice now that, we have the upper estimate $\|A_hz\|_{L^2(\Omega)}^2\lesssim\|\nabla{z}\|_{L^2(\Omega)}^2 + \|z\|_{L^2(\Omega)}^2$\footnote{By $a \lesssim b$ we mean that there exists a constant $C>0$ -- possibly depending on any \textit{fixed} quantity of the problem: $c, b, \kappa_0, \kappa_1$ and $\max\limits_{x \in \overline{\Gamma}} |h(x)|$ -- but independent of time and $\gamma$.}, which combined with Peter-Paul's inequality implies
		\begin{align}\label{est1-1}
			\left|\int_s^{T-s}\!\!\!\!\int_{\Omega} \gamma u_{tt} A_h(z) d\Omega dt\right|
			&\lesssim \dfrac{b}{2}\int_s^{T-s}\!\!\!\!\int_{\Omega}\varepsilon|\nabla{z}|^2 d\Omega dt
			\nonumber \\ &+ C_\varepsilon\left[\int_s^{T-s}\!\!\!\!\int_{\Omega}\gamma|u_{tt}|^2d \Omega dt + lot_\delta(z) \right]
		\end{align}
	for $\varepsilon>0$ to be precised later. Analogously, we deal with the last integral in the RHS of \eqref{est1} as follows
	\begin{equation}\label{est1-2}
		\left| \dfrac{b}{2} \int_s^{T-s}\!\!\!\!\int_{\Omega} z \nabla z \nabla(\mbox{div}(h)) d\Omega dt \right|
		\lesssim \dfrac{b}{2}\int_s^{T-s}\!\!\!\!\int_{\Omega}\varepsilon|\nabla{z}|^2 d\Omega dt + C_\varepsilon lot_\delta(z).
	\end{equation}
	Plugging \eqref{est1-1} and \eqref{est1-2} into \eqref{est1} and choosing $\varepsilon<J(h)/2$ we conclude the following upper estimate for the potential energy of $z$
	\begin{equation}\begin{aligned}\label{potz}
			\dfrac{b}{2}\int_s^{T-s}\!\!\!\!\int_{\Omega}|\nabla{z}|^2 d\Omega dt
			\lesssim [E_1(s)+E_1(T)] 
			+ b\int_{\Sigma_1}\kappa_1|z_t|^2 d\Sigma_1
		\\	+ \gamma\int_Q|u_{tt}|^2 dQ 
			+ \dfrac{b}{2}\int_s^{T-s}\!\!\!\!\int_{\Gamma_1}|\nabla z|^2(h\cdot\nu) d\Gamma_1 dt \\ + b\int_s^{T-s}\!\!\!\!\int_{\Gamma} \partial_\nu zA_h(z)d\Gamma dt
			+ \int_Q f^2 dQ
			+ lot_\delta(z).
	\end{aligned}\end{equation}
	In order to obtain the estimate for the kinetic part of the energy, we multiply \eqref{zprob} by $z$ and integrate by parts over $(s,T-s)\times\Omega$ to obtain 
	\begin{align}
		&\int_s^{T-s}\!\!\!\!\int_{\Omega} \left[b|\nabla{z}|^2 - z_t^2\right]d\Omega dt + \int_s^{T-s}\!\!\!\!\int_{\Gamma_0} \kappa_0z^2 d\Gamma_0dt = \int_s^{T-s}\!\!\!\!\int_{\Omega} fz d\Omega dt \nonumber \\ 
		& - \int_s^{T-s}\!\!\!\!\int_{\Omega} \gamma u_{tt}z d\Omega dt - \int_\Omega z_tzd\Omega \biggr\rvert_s^{T-s}+\dfrac{b}{2}\int_{\Gamma_1}\kappa_1z^2d\Gamma_1\biggr\rvert_s^{T-s}. \label{zmul} 
	\end{align}
	Identity \eqref{zmul} implies the following upper estimate for the kinetic energy
	\begin{equation}\begin{aligned}\label{kinz}
			\frac{1}{2}\int_s^{T-s}\!\!\!\!\int_{\Omega}|z_t|^2 d\Omega dt 
			\lesssim \frac{b}{2}\int_s^{T-s}\!\!\!\!\int_{\Omega}|\nabla{z}|^2 d\Omega dt + \frac{b}{2}\int_s^{T-s}\!\!\!\!\int_{\Gamma_0}\kappa_0|z|^2d\Gamma_0
		\\	+ [E_1(s)+E_1(T)] 
			+ b\int_{\Sigma_1}\kappa_1|z_t|^2d\Sigma_1
			+ \int_{Q}\gamma u_{tt}^2dQ 
			+ \int_Q f^2 dQ
			+ lot_\delta(z).
	\end{aligned}\end{equation}
	Combining \eqref{potz} with  \eqref{kinz} and accounting for \eqref{norm} we conclude
	\begin{equation}\begin{aligned}\label{e1est}
			\int_s^{T-s} E_1(t) dt
			\lesssim [E_1(s)+E_1(T-s)] 
			+ b\int_{\Sigma_1}\kappa_1|z_t|^2d\Sigma_1
			\\ + \int_{Q}\gamma u_{tt}^2dQ 
			+ lot_\delta(z)
			+ \Upsilon(s,T-s)
			+ \int_Q f^2 dQ
	\end{aligned}\end{equation}
	where 
	$$\Upsilon(s,T-s) = \dfrac{b}{2}\int_s^{T-s}\!\!\!\!\int_{\Gamma_1}|\nabla z|^2(h\cdot\nu) d\Gamma_1 dt + b\int_s^{T-s}\!\!\!\!\int_{\Gamma} \partial_\nu zA_h(z)d\Gamma dt$$
	and the boundary integral $\int_{\Gamma_0} |z|^2 d \Gamma_0$ resulting from \eqref{norm} is included in $lot_{\sigma}(z) $. The latter is by virtue of trace estimate and  compact embedding $H^{1} (\Omega)  \subset H^{1/2}(\Omega).$
	The boundary integrals above are estimated next. Recalling the notation  $A_h(z)= h\cdot\nabla z + \dfrac{1}{2}z\dvv{h}$ we have
	\begin{equation}\begin{aligned}\label{best1}
			\Upsilon(s,T-s)
			&= \frac{b}{2}\int_\alpha^{T-s}\int_{\Sigma_1}|\nabla{z}|^2(h\cdot\nu)d\Sigma_1dt
			+ b\int_s^{T-s}\int_{\Sigma}\partial_\nu{z}(h\cdot\nabla{z}) d\Sigma dt \\
			&\quad\quad+ \frac{b}{2}\int_s^{T-s}\int_{\Sigma}\partial_\nu{z}z\mbox{div}(h) d\Sigma dt\\
			&\equiv \frac{b}{2} I_1 + b I_2 + \frac{b}{2}I_3.
	\end{aligned}\end{equation}
	Moreover, we write the gradient at boundary as
	\begin{equation*}
		\nabla{z} = (\underbrace{\nabla{z}\cdot\nu}_{\partial_\nu{z}})\nu + \sum_{i=1}^{n-1}(\underbrace{\nabla{z}\cdot\tau_i}_{\partial_{\tau_i}{z}})\tau_i = (\partial_\nu{z})\nu + \sum_{i=1}^{n-1}(\partial_{\tau_i}{z})\tau_i,
	\end{equation*}
	which implies $|\nabla{z}|^2 = |\partial_\nu{z}|^2 + |\partial_\tau{z}|^2$. These along with the boundary conditions for the $z$--equation allow us to estimate the first integral in the RHS of \eqref{best1} as follows
	\begin{equation}\label{I1_1}
		|I_1| \lesssim  \int_{\Sigma_1}\kappa_1|z_t|^2 d\Sigma_1
		+ \int_s^{T-s}\|\partial_\tau{z}\|_{L^2(\Gamma_1)}^2 dt.
	\end{equation}
	
	The second integral $I_2$ is more involved. We first rewrite it as
	\begin{equation}\begin{aligned}\label{I2}
			I_2 
			&= \int_s^{T-s}  \int_{\Gamma_1} \partial_\nu{z}(h\cdot\nabla{z}) d\Gamma_1 dt \\
			&= \int_s^{T-s}\int_{\Gamma_1} |\partial_\nu{z}|^2(h\cdot\nu) d\Gamma_1 dt
			+ \int_s^{T-s} \left(\int_{\Gamma_1}+\int_{\Gamma_0}\right)\partial_\nu{z}\partial_{\tau}{z}(h\cdot\tau) d\Sigma dt \\
			&\equiv I_{2,1} + I_{2,2} + I_{2,3}
	\end{aligned}\end{equation} and then estimate the three resulting boundary integral terms. We notice that the boundary conditions for $z$ on $\Gamma_1$ allow us to estimate $I_{2,1}$ as
	\begin{equation}\label{I21}
		I_{2,1} \lesssim \int_{\Sigma_1}\kappa_1|z_t|^2d\Sigma_1.
	\end{equation}
	and $I_{2,2}$ as
	\begin{equation}\begin{aligned}\label{I22}
			I_{2,2}=\int_s^{T-s}\int_{\Gamma_1} \partial_\nu{z}\partial_\tau{z}(h\cdot\tau)d\Gamma_1dt
			\lesssim \int_{\Sigma_1} \kappa_1|z_t|^2 d\Sigma_1 + \int_s^{T-s}\|\partial_\tau{z}\|_{L^2(\Gamma_1)}^2 dt.
	\end{aligned}\end{equation} 
	
	For the remaining integral $I_{2,3}$ we use the boundary condition for $z$ on $\Gamma_0$: $\partial_\nu{z}=-z$ and the trace theorem, which implies that $z|_{\Gamma_0}\in H^{1/2}(\Gamma_0)$. On the other hand, since $z\in H^1(\Omega)$ we have $\partial_\tau{z}\in H^{-1/2}(\Omega)$. Hence, for $\varepsilon>0$ we have
	\begin{equation}\begin{aligned}\label{I23}
			I_{2,3} = \int_s^{T-s}\int_{\Gamma_0}\partial_\nu{z}\partial_{\tau}{z}(h\cdot\tau) d\Gamma_0 dt
			\\ \lesssim \int_0^T\left[C_\varepsilon\|\partial_\nu{z}\|_{\delta,\Gamma_0}^2 + \varepsilon\|\partial_\tau{z}\|_{-\delta,\Gamma_0}^2 \right]dt \\
			\lesssim C_{\varepsilon} \int_0^T \|z\|_{\delta,\Gamma_0}^2 dt 
			+ \varepsilon \int_0^T \|z\|_{3/2-\delta,\Omega}^2 dt \\ 
			\lesssim C_\varepsilon \int_0^T \|z\|_{H^{1/2+\delta}(\Omega)}^2dt + \varepsilon \int_0^T \| z\|_{H^1(\Omega)}^2dt \\ 
			\lesssim \varepsilon \int_0^T \|\nabla z\|_{L^2(\Omega)}^2dt + C_T lot_\delta(z)
	\end{aligned}\end{equation}
	for $\varepsilon>0$ to be determined. Here also, the boundary term in (\ref{norm}) is included in a  lower order term. 
	
	Plugging \eqref{I21}, \eqref{I22} and \eqref{I23} into \eqref{I2} we estimate $I_2$ as
	\begin{equation}\label{I2_1}
		|I_2| 
		\lesssim \int_{\Sigma_1} \kappa_1|z_t|^2 d\Sigma_1
		+ \int_s^{T-s}\|\partial_\tau{z}\|_{L^2(\Gamma_1)}^2 dt
		+ \varepsilon \int_0^T \|\nabla z\|_{L^2(\Omega)}^2dt + C_T lot_\delta(z)
	\end{equation}
	Finally, integral $I_3$ is estimated as
	\begin{equation}\label{I3}
		|I_3|
		\lesssim \int_s^{T-s}\int_{\Gamma} |\partial_\nu{z}|^2 + |z|^2 d\Gamma dt 
		\lesssim \int_{\Sigma_1} \kappa_1|z_t|^2 d\Sigma_1 dt + lot_\delta(z).
	\end{equation}
	
	Collecting \eqref{I1_1}, \eqref{I2_1} and \eqref{I3} and returning to \eqref{best1} we conclude
	\begin{equation}\begin{aligned}\label{best2}
			|\Upsilon(s,T-s)|
			\lesssim 
			\int_{\Sigma_1}\kappa_1|z_t|^2 d\Sigma
			+ \int_s^{T-s}\|\partial_\tau{z}\|_{L^2(\Gamma_1)}^2 dt 
			\\+ \varepsilon \int_0^T \|\nabla z\|_{L^2(\Omega)}^2dt 
			+ C_T lot_\delta(z).
	\end{aligned}\end{equation}
	The tangential derivative above is estimated using an adaptation of Lemma 2.1 in \cite{LL_2002_NA}, which was obtained for the homogeneous case,
	\begin{equation}\begin{aligned}\label{tang1}
		\int_s^{T-s} \|\partial_\tau{z}\|_{0,\Gamma_1}^2 dt
		\leqslant \\ C_T \left[ \int_0^T \|\partial_\nu{z}\|_{0,\Gamma}^2 + \|z_t\|_{0,\Gamma_1}^2 dt + lot_\delta(z) + \|\gamma u_{tt} + f\|_{H^{-1/2+s}(Q)}^2 \right].\end{aligned}
	\end{equation}
	Combining \eqref{best2} and \eqref{tang1} we arrive at
	\begin{equation}\label{best3}\begin{split}
		|\Upsilon(s,T-s)|
		\lesssim \varepsilon \int_0^T \|\nabla z\|_{L^2(\Omega)}^2dt
		+ C_T\left[ \int_{\Sigma_1}\kappa_1|z_t|^2d\Sigma_1 
		+ lot_\delta(z) \right. \\ \left. + \|\gamma u_{tt} + f\|_{H^{-1/2+s}(Q)}^2 \right]
	\end{split}\end{equation}
	Returning with \eqref{best3} to \eqref{e1est} and choosing $\varepsilon>0$ properly we conclude \eqref{e1rec}.
\end{proof}

Our next result deals with $lot_\delta(z)$, which can be absorbed by the damping using a compactness uniqueness argument.
\begin{proposition}\label{compuniq}
	For $T>0$ there exists a constant $ C_T > 0 $ such that  the following inequality holds:
	\begin{equation}
		lot_\delta(z)
		\leq C_T\left[ b\int_{\Sigma_1} \kappa_1|z_t|^2 d\Sigma_1 
		+ \int_Q \gamma|u_{tt}|^2 dQ \right].
	\end{equation}
\end{proposition}

\begin{proof}
	As pointed out in \eqref{e1rec}, we have $$lot_\delta(z)\leq C_\delta\sup_{t\in[0,T]}\left\{\|z\|_{H^{1-\delta}(\Omega)}^2+\|z_t\|_{H^{-\delta}(\Omega)}^2\right\},$$for $\delta \in (0,1/2)$. Then we prove Proposition \eqref{compuniq} as a corollary of the following Lemma \begin{lemma}
		\label{compunilem} For every $\delta \in (0,1/2) $,there exists a constant $C_{T,\delta} >0 $ such that \begin{equation}
			\label{cu1} \|(z,z_t)\|_{L^2(0,T; H^{1-\delta}(\Omega) \times H^{-\delta}(\Omega))}^2 \leqslant C_T\left[ b\int_{\Sigma_1} \kappa_1|z_t|^2 d\Sigma_1 
			+ \int_Q \gamma|u_{tt}|^2 dQ \right]
		\end{equation}
	\end{lemma}\begin{proof}
		Using the notation of \cite{simon1986compact}, let $X = H^1(\Omega)$, $B = H^{1-\delta}(\Omega)$ and $Y = H^{-\delta}(\Omega).$ Then it follows from \cite[Theorem 16.1]{lions_non-homogeneous_1972} that the injection of $X$ in $B$ is compact. Moreover, since $\delta \in (0,1/2)$, \cite[Theorem 12.4]{lions_non-homogeneous_1972} allows us to write $$Y = H^{-\delta}(\Omega) = [L^2(\Omega),H^{-1}(\Omega)]_\delta,$$ and then the injection of $B$ in $Y$ is continuous (even dense). Introduce the space $\Lambda$ as $$\Lambda \equiv \{v \in L^2(0,T;X); \dot{v} \in L^2(0,T; Y)\}$$ equipped with the norm $$\|v\|_{\Lambda} = \|v\|_{L^2(0,T;X)} + \|\dot{v}\|_{L^2(0,T;Y)}.$$ Then it follows from \cite{simon1986compact} that the injection of $W$ into $L^2(0,T;B)$ is compact. We are then ready for proving \eqref{cu1}
		
		By contradiction, suppose that there exists a sequence of initial data $\{u_{0n},u_{1n},u_{2n}\}$ with corresponding $E_1^n(0)$ energy uniformly (in $n$) bounded generating  a sequence $\{u_n, \dot{u}_n, \ddot{u}_n\}$ of solutions of problem \eqref{usist} with related 
		sequence $$\left\{z_n = \dfrac{c^2}{b} u_n + \dot{u}_n, \dot{z}_n = \dfrac{c^2}{b}\dot{u}_n + \ddot{u}_n\right\}$$ solutions of problem \ref{usistz} such that 
		\begin{subnumcases}{}
			\|z_n\|_{L^2(0,T; H^{1 - \delta}(\Omega))}^2 +  \|\dot{z}_n\|_{L^2(0,T; H^{- \delta}(\Omega))}^2 \equiv 1 \label{E10-3}	 \\[2mm]
			\dfrac{c^2}{b} \int_0^T \int_\Omega \gamma(\ddot{u}_n)^2 dQ + \int_0^T \int_{\Gamma_1}\kappa_1 (\dot{z}_n)^2d\Sigma_1 \to 0, \ \mbox{as} \ n \to +\infty. \label{E10-4}
		\end{subnumcases}
		
		From idenity \eqref{e1id} (with $f = 0$) we see that the uniform boundedness $E_1^n(0)$ implies uniform boundedness of $E_1^n(t)$, $ t \in [0,T]$. Therefore, one might choose a (non--relabeled) subsequence satisfying\begin{subequations}\begin{align}
				z_{n} \to \ \mbox{some } \zeta, \ \mbox{weak}^\ast \ \mbox{in } L^\infty(0,T;H^1(\Omega)) 
				\label{E10-72a} \\ \dot{z}_{n} \to \ \mbox{some } \zeta_1, \ \mbox{weak}^\ast \ \mbox{in } L^\infty(0,T; L^2(\Omega)) \hookrightarrow L^2(0,T; H^{ - \delta}(\Omega)); \label{E10-72b} \\ 
				\gamma^{1/2}\dot{u}_{n} \to \ \mbox{some } \eta, \ \mbox{weak}^\ast \ \mbox{in } L^\infty(0,T;L^2(\Omega)); \label{E10-72c} 
		\end{align}\end{subequations}  It easily follows from distributional calculus that $\dot{\zeta} = \zeta_1$ and, in the limit, the functions $\zeta$ and $\eta$ satisfy the equation \begin{subnumcases}{}
			\ddot{\zeta} = b\Delta \zeta -\gamma^{1/2} \dot{\eta} \hspace{7cm} \mbox{in} \ Q  \label{E10-9a} \\[2mm]
			\gamma^{1/2}\dot{\zeta} = \dfrac{c^2}{b}\eta + \dot{\eta} \label{E10-9aa} \\[2mm]
			\left[\dfrac{\partial \zeta}{\partial \nu}+ \kappa_1\dot{\zeta}\right]\biggr\rvert_{\Sigma_1} = 0 ; \qquad \left[\dfrac{\partial \zeta}{\partial \nu}+ \kappa_0\zeta\right]\biggr\rvert_{\Sigma_0} = 0. \label{E10-9c} 
		\end{subnumcases} plus respective initial data.
		
		It follows from the weak convergence that there exist $M$ independent of $n$ such that \begin{equation}
			\|(z_n, \dot{z}_n)\|_{L^\infty(0,T; H^{1}(\Omega) \times H^{-\delta}(\Omega))} = \|z_n\|_\Lambda \leqslant \ M,\label{E10-17}
		\end{equation} for all $n$. Then, by \emph{compactness} (of $\Lambda$ in $L^2(0,T; H^{1-\delta}(\Omega))$ there exists a subsequence, still indexed by $n$, such that \begin{equation}
			z_n \to \zeta \ \mbox{strongly in } L^2(0,T; H^{1 - \delta}(\Omega)). \label{E10-21}
		\end{equation} 
	
		Next we show that $\eta$ and $\zeta$  are zero elements. Indeed, from (\ref{E10-4}) 
		we obtain that  
		$\gamma^{1/2}\ddot{u}_n \to 0$ in $L^2(0,T;L^2(\Omega) $ and
		$\dot{z}_n\rvert_{\Gamma_1}  \to 0 $ in $ L^2(0,T;L^2(\Gamma_1)).$ This implies that
		$\dot{\eta} =0 $ and $\dot{\zeta}|_{\Gamma_1} =0$. 
		Indeed, the last claim follows from 
		${\gamma^{1/2}} \ddot{u}_n \rightarrow \dot{\eta}$ in $H^{-1} (0,T;L^2(\Omega) $
		where by the uniqueness of the limit one must have $\dot{\eta} \equiv 0$ .
		Similar argument applies to infer $\dot{\zeta}|_{\Gamma_1} =0.$
		
		Next, passing to the limit as $n \to \infty$  yields the following over determined (on $\Gamma_1$) problem:
		\begin{subnumcases}{}
			\ddot{\zeta} = b\Delta \zeta  \hspace{7cm} \mbox{in} \ Q  \label{E10-9a2} \\[2mm]
			\gamma^{1/2}\dot{\zeta} = \dfrac{c^2}{b}\eta  \label{E10-9aa2} \\[2mm]
			\left[\dfrac{\partial \zeta}{\partial \nu}\right]\biggr\rvert_{\Sigma_1} = 0 ; \qquad \left[\dfrac{\partial \zeta}{\partial \nu}+ \kappa_0\zeta\right]\biggr\rvert_{\Sigma_0} = 0; \qquad \dot{\zeta}_t\rvert_{\Gamma_1} =0 \label{E10-9c2}
		\end{subnumcases} 
		plus respective initial data.
		
		The overdetermined $\zeta$--problem implies  in particular with $v\equiv \zeta_t $
		$$\ddot{v} = b \Delta v $$ with the overdetermined boundary conditions 	$$\dfrac{\partial v}{\partial  \nu}\biggr\rvert_{\Gamma_1} = 0; \qquad  v\rvert_{\Gamma_1} = 0$$
		which yields overdetermination of boundary  data on $\Gamma_1$ for the wave operator.
		This gives  $v \equiv 0 $, hence $\zeta_t \equiv  0$ and $\zeta_{tt} =0 $   distributionally .
		Using this information in (\ref{E10-9a}) yields
		$$\Delta \zeta =0; \qquad  \dfrac{\partial \zeta}{\partial \nu}\biggr\rvert_{\Gamma_1} =0;\qquad
		\left[\dfrac{\partial \zeta}{\partial \nu} + \kappa_0 \zeta\right]_{\Gamma_0}  =0.$$
		Standard elliptic estimate, along with $\kappa_0 > 0$  gives $\zeta \equiv 0 $ in  $Q.$
		
		Finally, weak$^\ast$ convergence of $z_n$ in $L^\infty(0,T; L^2(\Omega)$ and the compacity of $L^2(\Omega)$ into $H^{-\delta}(\Omega)$ (see \cite[Theorem 16.1 with $s = 0$ and $\varepsilon = \delta$]{lions_non-homogeneous_1972})[so $\dot{z}_n(t) \to \dot{\zeta}(t)$ strongly in $H^{-\delta}(\Omega)$ for a.e. $t \in [0,T]$] allow us to compute (due to Lebesgue dominated convergence theorem
		): 
		\begin{align*}
			\lim\limits_{n \to \infty} \|z_n\|_{L^2(0,T; H^{-\delta}(\Omega)}^2 = \lim\limits_{n \to \infty} \int_0^T\|z_n(t)\|_{H^{-\delta}(\Omega)}^2dt =  \int_0^T\lim\limits_{n \to \infty}\|z_n(t)\|_{H^{-\delta}(\Omega)}^2dt \\
			= \|\dot{\zeta}\|_{L^2(0,T; H^{-\delta}(\Omega)}^2 = 0
		\end{align*} since $\dot{\zeta} \equiv 0$ in $Q.$ Then, passing with the limit as $n \to \infty$ in \eqref{E10-3} we have $$0 = \|z\|_{L^2(0,T;H^{1-\delta}(\Omega))} = 1,$$ which is a contradiction. The Lemma is proved.
	\end{proof}
Lemma \ref{compunilem}  implies in a straightforward way the result of the  proposition \ref{compuniq}.
\end{proof}

We are ready to establish the exponential decay of the the energy functional $E_1$.
\begin{theorem}\label{thm34}
	Assume that $f=0$. Hence, the energy functional $E_1$ is exponentially stable, i.e. there exists $T>0$ and constants $M,\omega>0$ such that 
	\begin{equation}\label{e1exp}
		E_1(t) \leq M e^{-\omega t} E_1(0),\quad\mbox{for}~t>T.
	\end{equation}
\end{theorem}
\begin{proof}
	Using identity \eqref{e1id} we have
	\begin{equation*}
		\left(\int_0^s + \int_{T-s}^T\right) E_1(t) dt \leq 2sE_1(0).
	\end{equation*}
	Since $s<T/2$ can be taken arbitrarily small, we fix $s<1/2$ in the above inequality and use it to completethe $L^1$--norm of the energy $E_1$ in \eqref{e1rec}. We obtai
	\begin{equation}\begin{aligned}
			\int_0^T E_1(t) dt 
			\lesssim E_1(0) + E_1(s) + E_1(T-s) 
			\\ + C_T\left[ \int_{\Sigma_1} b\kappa_1 z_t^2 d\Sigma_1 
			+ \int_Q \gamma u_{tt}^2 dQ + lot_\delta(z)\right].
	\end{aligned}\end{equation}
	The remaining terms in $s$ are estimated using the dissipativity of $E_1$ (see identity \eqref{e1id} for $f=0$). In fact we rewrite the above as follows
	\begin{equation*}
		\int_0^T E_1(t) dt 
		\lesssim E_1(T) + C_T \left[ \int_{\Sigma_1}b\kappa_1 z_t^2 d\Sigma_1 + \int_Q \gamma u_{tt}^2 dQ + lot_\delta(z) \right].
	\end{equation*}
	The $lot_\delta(z)$ is ``absorved'' using Lemma \ref{compuniq}, thus
	\begin{equation}\label{e1xp_1}
		\int_0^T E_z(t) dt
		\lesssim E_1(T) + C_T\left[ \int_{\Sigma_1}b\kappa_1 z_t^2 d\Sigma_1 + \int_Q \gamma u_{tt}^2 dQ \right].
	\end{equation}
	On the other hand, using identity \eqref{e1id} (with $f=0$) once more, we deduce
	\begin{equation}\begin{aligned}\label{e1xp_2}
			TE_1(T)
			\lesssim \int_0^TE_1(t)dt + C_T\left[ \int_{\Sigma_1} b\kappa_1 z_t^2 d\Sigma_1 + \int_Q \gamma u_{tt}^2 dQ \right].
	\end{aligned}\end{equation}
	Combining \eqref{e1xp_1} and \eqref{e1xp_2} we arrive at
	\begin{equation*}
		(T-C)E_1(T) + \int_0^T E_1(t) dt
		\leq C_T \left[ \int_{\Sigma_1} b\kappa_1 z_t^2 d\Sigma_1 + \int_Q \gamma u_{tt}^2 dQ 
		\right],
	\end{equation*}
	for some $C>0$. Choosing $T=2C$ and replacing the ``damping'' term using identity \eqref{e1id} (with $f=0$) we rewrite the above estimate as follows
	\begin{equation*}
		E_1(T) + \int_0^T E_1(t) dt
		\lesssim C_T[E_1(0) - E_1(T)]
	\end{equation*}
	which implies
	\begin{equation*}
		E_1(T) \leq \underbrace{\frac{C_T}{1+C_T}}_{\mu} E_1(0),
	\end{equation*}
	where $0<\mu<1$ does not depend on the solution. Repeating the process on the interval $[mT,(m+1)T]$ and we obtain $E_1((m+1)T)\leq\mu E_1(mT)$, for every $m\geq 0$. This implies
	\begin{equation*}
		E_1(mT) \leq \mu^m E_1(0),
	\end{equation*}
	for every $m\geq 1$. Thus, for $t>T$ we write $t = mT + s$, with $s\in(0,T]$ and $m\geq1$, which implies
	\begin{equation*}
		E_1(t) \leq E_1(mT) 
		\leq \mu^m E_1(0) = e^{-|\ln{\mu}|m}E_1(0) = e^{-|\ln{\mu}|\frac{t-s}{T}} E_1(0) = \frac{1}{\mu}e^{-\frac{|\ln{\mu}|}{T}t}E_1(0),
	\end{equation*}
	which implies \eqref{e1exp} with $\omega = |\ln{\mu}|/T$ and $M=1/\mu$.
\end{proof}

The previous result is key to establish the exponential stability of $E(t)$, which is given next.

\begin{proof}[\bf Proof of Theorem \ref{thm34n}] Notice that the exponential decay for $E_1$ obtained in Theorem \eqref{thm34} implies exponential decay of the quantities $\|z\|_{\mathcal{D}(A^{1/2})}, \|z_t\|_{L^2(\Omega)}$, and we will show that this implies exponential decay of $E$. In view of Lemma \ref{equinorm}, the only remaining quantity we need to show exponential decay is $\| u\|_{\mathcal{D}(A^{1/2})}$ and this follows from the fact that $bu_t + c^2u = z$. Indeed, the variation of parameter formula implies that \begin{equation}
		\label{NE4-35} u(t) = e^{-\frac{c^2}{b}t}u_0 + \int_0^t e^{-\frac{c^2}{b}(t-\tau)}z(\tau)d\tau,
	\end{equation} then, computing the $\mathcal{D}(A^{1/2})$--norm both sides we estimate \begin{equation}
		\|u(t)\|_{\mathcal{D}(A^{1/2})} \leqslant e^{-\frac{c^2}{b} t} \|u_0\|_{\mathcal{D}(A^{1/2})} + \int_0^t e^{-\frac{c^2}{b}(t-\tau)}\|z(\tau)\|_{\mathcal{D}(A^{1/2})}d\tau \label{E77nn-15}
	\end{equation} hence it follows from \eqref{e1exp} that \begin{align*}
		\|u(t)\|_{\mathcal{D}(A^{1/2})} &\leqslant e^{-\frac{c^2}{b} t} \|u_0\|_{\mathcal{D}(A^{1/2})} + ME_1(0)\int_0^t e^{-\frac{c^2}{b}(t-\tau)-\omega\tau}d\tau \nonumber \\ &\leqslant e^{-\frac{c^2}{b} t} E(0) + \dfrac{(c^2 -b\omega)(e^{-\omega t}-e^{-\frac{c^2}{b}t})}{\omega c^2}ME(0)\leqslant \overline{M}e^{-\overline{\omega}t}E(0).
	\end{align*} where we have made the benign assumption that $\dfrac{c^2}{b} > \omega$ from \eqref{e1exp}, as if $\omega \geqslant \dfrac{c^2}{b}$ we use formula \eqref{e1exp} with $\omega_1 := \dfrac{c^2}{b} - \varepsilon$ so $\omega > \omega_1$ and $\dfrac{c^2}{b} > \omega_1$.
	
	The proof is complete.
\end{proof}

\section*{Acknowledgments} The research of I. L. was partially supported by the National Science Foundation under Grant DMS-1713506. The work was partially carried out while I. L. was member of the MSRI program “Mathematical Problems in Fluid Dynamics” during the Spring 2021 semester (NSF DMS-1928930) of the University of California, Berkeley; while M. B. was a member of the Weierstrass Institute for Applied Analysis and Stochastics, Berlin, Germany. The authors thank their host organizations. 

\bibliographystyle{abbrvurl} 
\bibliography{ref2.bib}

\end{document}